\documentclass[12pt,reqno]{amsart}

 \usepackage{mathrsfs}

\usepackage{color}
\usepackage{amsmath, amsthm, amssymb}
\usepackage{amsfonts}
\usepackage[dvips]{epsfig}
\usepackage{graphicx}
\usepackage{caption}
\usepackage{subcaption}
\usepackage[english]{babel}
\usepackage{hyperref}

\usepackage{tikz}
\usepackage{rotating}
\usepackage[utf8]{inputenc}
\usepackage{cite}
\usepackage{amscd}
\usepackage{color}
\usepackage{bm}
\usepackage{enumerate}

\usepackage{verbatim}
\usepackage{hyperref}
\usepackage{amstext}
\usepackage{latexsym}
\theoremstyle{plain}
\newtheorem{theorem}{Theorem}[section]

\newtheorem{proposition}[theorem]{Proposition}
\newtheorem{lemma}[theorem]{Lemma}

\numberwithin{theorem}{section}
\numberwithin{equation}{section}

\newcommand{\average}{{\mathchoice {\kern1ex\vcenter{\hrule height.4pt
width 6pt depth0pt} \kern-9.7pt} {\kern1ex\vcenter{\hrule
height.4pt width 4.3pt depth0pt} \kern-7pt} {} {} }}

\def\R{\mathbb{R}}

\renewcommand{\a }{\alpha }
\renewcommand{\b }{\beta }
\renewcommand{\d}{\delta }

\newcommand{\D }{\Delta }
\newcommand{\tr }{\hbox{ tr } }
\newcommand{\e }{\varepsilon }
\newcommand{\g }{\gamma}

\newcommand{\G }{\Gamma}

\newcommand{\n }{\nabla }

\renewcommand{\phi}{\varphi}

\newcommand{\s }{\sigma }

\renewcommand{\t }{\tau }

\renewcommand{\th }{\theta }

\renewcommand{\O }{\Omega }

\newcommand{\ov}{\overline}

\newcommand{\be}{\begin{equation}}
\newcommand{\ee}{\end{equation}}

\newcommand{\de}{\partial}

\newcommand{\ti}{\widetilde}

\renewcommand{\k}{\kappa}

\newcommand{\calO }{\mathcal{O}}
\newcommand{\calC }{\mathcal{C}}

\newcommand{\calD }{\mathcal{D}}

\newcommand{\N}{\mathbb{N}}

\newcommand{\cD}{{\mathcal D}}

\newcommand{\cR}{{\mathcal R}}

\newcommand{\B}{{Q}}

\renewcommand{\epsilon}{\varepsilon}

\begin{document}
 
\title[Influence of an $L^p$-perturbation on Hardy-Sobolev inequality with singularity a curve]
{Influence of an $L^p$-perturbation on Hardy-Sobolev inequality with singularity a curve}
\author{IDOWU ESTHER IJAODORO}
\address{I. E. I.: African Institute for Mathematical Sciences in Senegal, KM 2, Route de
Joal, B.P. 14 18. Mbour, Senegal.}
\email{idowu.e.ijaodoro@aims-senegal.org}
\author{El Hadji Abdoulaye THIAM}
\address{H. E. A. T. : Université de Thies, UFR des Sciences et Techniques, département de mathématiques, Thies.}
\email{elhadjiabdoulaye.thiam@univ-thies.sn}
\begin{abstract}
We consider a  bounded domain $\Omega$ of $\mathbb{R}^N$, $N\ge3$, $h$ and $b$ continuous functions on $\Omega$. Let $\Gamma$ be  a closed curve contained in $\Omega$. We study existence of positive solutions $u \in H^1_0\left(\Omega\right)$ to the perturbed Hardy-Sobolev equation:
$$
-\Delta u+h u+bu^{1+\delta}=\rho^{-\sigma}_\Gamma u^{2^*_\sigma-1} \qquad \textrm{ in } \Omega,
$$
where $2^*_\sigma:=\frac{2(N-\sigma)}{N-2}$ is the critical Hardy-Sobolev exponent, $\sigma\in [0,2)$, $0< \delta<\frac{4}{N-2}$ and $\rho_\Gamma$ is the distance function to $\Gamma$. We show that the existence of minimizers does not depend on the local geometry of $\Gamma$ nor on the potential $h$. For $N=3$, the existence of ground-state solution may depends on the trace of the regular part of the  Green function of $-\Delta+h$ and or on $b$. This is due to the perturbative term  of order ${1+\d}$.
\end{abstract}
\maketitle
\bigskip
\noindent
\textit{AMS Mathematics Subject Classification:} 35J91, 35J20, 35J75.\\\
\noindent\\\
\textit{Key words}: Hardy-Sobolev inequality; Positive minimizers; Parametrized curve; Mass; Green function.

\section{Introduction}\label{Intro}
Hardy-Sobolev inequality with a cylindrical weight states, for  $N\geq 3$, $0\leq k\leq N-1$ and $\s\in[0,2)$, that
\begin{equation}\label{Hardy-Sobolev}
\int_{\R^N} |\nabla v|^2 dx \geq C  \biggl(\int_{\R^N} |z|^{-\s} |v|^{2^*_\s} dx\biggl)^{2/2^*_\s}\qquad \textrm{ for all $v \in \mathcal{D}^{1,2}({\R^N})$,}
\end{equation}
where $x=(t,z)\in \R^k \times \R^{N-k}$,  $C= C(N,\s,k)>0$, $  2^*_\s:= \frac{2(N-\s)}{N-2}$ is the critical Hardy-Sobolev exponent and   $\calD^{1,2}(\R^N)$ is the completion of $C^\infty_c(\R^N)$ with respect to the norm 
$$
v\longmapsto\left(\int_{\R^N} |\n v|^2 dx\right)^{1/2}.
$$
Inequality \eqref{Hardy-Sobolev} can be obtained by interpolating between Hardy (which corresponds to the case $\s=2$ and $k\not=N-2$) and Sobolev (which is  the case $\s=0$) inequalities. This inequality is invariant by scaling  on $\R^N$ and by translations in the $t$-direction.\\

When $\s=2$ and $k \not=N-2$, the best constant is $\left(\frac{N-k-2}{2}\right)^2$ but it is never achieved. For $\s\in [0, 2)$, the best constant $C$ in \eqref{Hardy-Sobolev} is given by  
\be\label{eq:Hard-Sob-sharp}
S_{N,\s}:=\min\left\{ \frac{1}{2}\int_{\R^N} |\nabla v|^2 dx-\frac{1}{2^*_\s}\int_{\R^N} |z|^{-\s} |v|^{2^*_\s} dx, \,\, v\in \cD^{1,2}(\R^N)   \right\}.
%
\ee
In the case $\s\in [0,2)$ and $k=0$, $S_{N,\s}$ is achieved by the standard bubble $c_{N,\s} (1+|x|^{2-\s})^{\frac{2-N}{2-\s}}$, see for instance Aubin \cite{Talenti}, Talenti \cite{Aubin} and Lieb  \cite{Lieb}.  When $k=N-1$, the support of the minimizer is contained in a half space, see Musina \cite{Musina}. 
For $1\leq k\leq N-2$ and $\s\in (0,2)$, Badiale and Tarentello \cite{BT} proved the existence of a minimizer $w$ for \eqref{eq:Hard-Sob-sharp}. They were motivated by questions from astrophysics. Later Mancini, Fabbri and Sandeep used the moving plane method to prove that $w(t,z)=\th(|t|, |z|) $, for some positive function $\th$.    An interesting classification result was also derived in \cite{FMS} when $\s=1$,  that every minimizer is of the form $ c_{N,k} ((1+|z|)^2+|t|^2)^{\frac{2-N}{2}}$, up to scaling in $\R^N$ and translations in the $t$-direction.\\
Since in this paper we are interested with Hardy-Sobolev inequality with  weight singular at a given curve, our asymptotic energy level is given by $S_{N,\s}$  with     $k=1$ and $\s\in [0,2)$.\\

Let $\O$ be a bounded domain in $\R^N$, $N\geq 3$, $h$ and $b$ continuous function on $\O$. Let  $\G \subset \O$ be a smooth closed curve. In this paper, we are concerned with the existence of minimizers for the infinimum
\be \label{eq:min-to-study}
\mu_\s(\O, \G, h, b):=\inf_{ u\in  H^1_0(\O) \setminus\{0\} } \displaystyle \frac{1}{2}\int_{\O} |\nabla u|^2 dx +\frac{1}{2}\int_\O h u^2 dx+\frac{1}{2+\d}\int_{\O} b u^{2+\d} dx -{\frac{1}{2^*_\s}}\int_{\O} \rho_\G^{-\s} |u|^{2^*_\s} dx,
\ee
where $\s\in[0,2)$,  $\displaystyle 2^*_\s:= \frac{2(N-\s)}{N-2}$, $0< \delta < \frac{4}{N-2}$ and $ \rho_\G(x):=\textrm{dist}(x,\G)$ is the distance function to $\Gamma$.  Here and in the following, we assume that $-\D+h$ defines a  coercive bilinear form on $H^1_0(\O)$ and that $b\leq 0$. We are interested with the effect of $b$ and/or the location of the curve $\G$ on  the existence of minimmizer for $\mu_\s(\O, \G, h, b)$.\\
When there is no perturbation, and $\s=0$, problem \eqref{eq:min-to-study} reduces to the famous Brezis-Nirenberg problem \cite{BN}. In this case,  for $N\geq 4$ it is enough that $h(y_0)<0$ to get a minimizer, whereas for $N=3$, the existence of minimizers is guaranteed by the positiveness of  a certain mass, see Druet  \cite{Druet}.\\
Here, we deal with the case $\s\in [0,2)$.Our first result deals with the case $N\geq 4$. Then we have
\begin{theorem}\label{th:main1}
Let $N \geq 4$, $\s\in [0,2)$ and   $\O$  be a   bounded domain of $\R^N$. Consider  $\G$ a smooth closed curve contained in $\O$.
Let  $h$ and $b$ be  continuous function such that the linear operator $-\D+h$ is coercive and $b \leq 0$.  We assume that
\begin{equation}\label{eq:h-bound-main-th-11}
b(y_0) < 0,
\end{equation}
for some $y_0 \in \G$. Then $\mu(\O, \G, h, b)$ is achieved by a positive function $u \in H^1_0(\O)$.
\end{theorem}
In contrast, to the result of the second author and Fall \cite{Fall-Thiam}, inequality \eqref{eq:h-bound-main-th-11} in Theorem \ref{th:main1} shows that there is no influence of the curvature of $\G$ nor the potential $h$. This is due to the influence of the added perturbation  term in \eqref{eq:min-to-study}.  \\
For $N=3$, we let   $G(x,y)$ be  the Dirichlet Green function of the operator $-\D +h$, with zero Dirichlet data. It satisfies 
\be\label{eq:Green-expan-introduction}
\begin{cases}
-\D_x G(x,y)+h(x) G(x,y)=0&  \qquad\textrm{  for every $x\in \O\setminus\{y\}$}\\\\
G(x,y)=0 &  \qquad\textrm{  for every $x\in\de  \O$.}
\end{cases}
\ee
In addition, there exists a continuous function   $\textbf{m}:\O\to \R$ and a positive constant $c>0$ such that  
\be \label{eq:expans-Green}
 G (x,y)=\frac{c}{  |x- y|}+ c\, \textbf{m}(y)+o(1)  \qquad \textrm{ as $x \to y.$} 
\ee
This  function   $\textbf{m}:\O\to \R$ is  the    \textit{mass} of $-\D+h$ in $\O$. Our second main result is the following
\begin{theorem}\label{th:main2}
Let $\s\in [0,2)$ and   $\O$  be a   bounded domain of $\R^3$. Consider  $\G$ a smooth closed curve contained in $\O$.
Let  $h$ and $b$ be continuous functions such that the linear operator $-\D+h$ is coercive and $b \leq 0$. We assume that
\begin{equation}\label{eq:h-bound-main-th-1}
\begin{cases}
b(y_0) < 0 &\qquad \textrm{ for \quad $2<\delta <4$}\\\\
m(y_0)>c b(y_0) &\qquad \textrm{ for \quad $\d=2$}\\\\
m(y_0)>0  & \qquad \textrm{ for \quad $0<\d<2$},
\end{cases}
\end{equation}
for some positive constant $c$ and $y_0\in \G$. Then $\mu_\s\left(\O,\G, h, b\right)$ is achieved by a positive function $u \in H^1_0(\O)$.  
\end{theorem}
%
%
%
%
%
%

The literature about Hardy-Sobolev  inequalities on domains  with various singularities is very hudge. The existence of minimizers depends on the curvatures at a point of the singularity. For more details, we refer to Ghoussoub-Kang \cite{GK},  Ghoussoub-Robert  \cite{GR3, GR5}, Demyanov-Nazarov \cite{DN}, Chern-Lin \cite{CL}, Lin-Li \cite{LL}, Fall-Thiam \cite{Fall-Thiam},  Fall-Minlend-Thiam in \cite{FMT} and the references there in. We refer also  to  Jaber \cite{Jaber1, Jaber4} and Thiam \cite{Thiam1, Thiam, ThiamH} and references therein, for  Hardy-Sobolev inequalities on Riemannian manifold. Here also the impact of the scalar curvature at the point singularity plays an important role for the existence of minimizers in higher dimensions $N\geq 4$. The paper  \cite{Jaber1} contains also existence result under positive mass condition for $N=3$. \\

The proof of Theorem \ref{th:main1} and Theorem \ref{th:main2} rely on test function methods. Namely to build appropriate test functions allowing to compare $\mu_\s(\O, \G, h, b)$ and $ S_{N,\s}$. We find a  continuous family  of test functions $(u_\e)_{\e>0}$ concentrating at a point $y_0\in \G$ which yields $\mu(\O, \G, h, b)<S_{N,\s}  $, as $\e\to 0$, provided \eqref{eq:h-bound-main-th-1} holds.   In Section \ref{s:3D-case}, we consider   the case $N=3$.  Due to the fact that   the ground-state $w$ for $S_{3,\s}$, $\s\in (0,2)$ is not known explicitly,   it is not radially symmetric, it is not smooth and $S_{3,\s}$ is only invariant under translations in the $t-$direction; we could only construct a discrete family of test functions $(\Psi_{\e_n})_{n\in \N}$ that leads to the inequality $\mu_\s(\O, \G, h, b)<S_{3,\s}$.  These are
similar to the  test  functions $(u_{\e_n})_{n\in \N}$ in dimension $N\geq 4$ near the concentration point $y_0$, but away from it is  substituted  with the regular part of the  Green function $G(x,y_0)$, which makes appear the mass $\textbf{m}(y_0)$ and/or $b(y_0)$ in its first order Taylor expansion, see \eqref{eq:expans-Green}.  \\

The paper is organized as follows: In Section \ref{Section1}, we recall some geometric and analytic preliminaries results relating to the local geometry of the curve $\G$ and the decay estimates of the ground state $w$ of $S_{N, \s}$. In Section \ref{s:Exits-N-geq4} and Section \ref{s:3D-case}, we construct a test function for $\mu_\s(\O, \G, h, b)$ in order to prove Theorem \ref{th:main1} and Theorem \ref{th:main2}. Their proof is completed in Section \ref{Section5}.

\bigskip
\noindent

\section{Preliminaries Results}\label{s:Geometric-prem}\label{Section1}
Let    $\G\subset \R^N$ be  a smooth closed  curve. Let $(E_1;\dots; E_N)$ be an orthonormal basis of $\R^N$.
 For $y_0\in \G$ and $r>0$ small,  we consider the curve $\gamma:\left(-r, r\right) \to \G$,   parameterized by arclength such that $\gamma(0)=y_0$. Up to a translation and a rotation,  we may assume that $\g'(0)=E_1$.     We choose a smooth   orthonormal frame field $\left(E_2(t);...;E_N(t)\right)$ on the normal bundle of $\G$ such that $\left(\g'(t);E_2(t);...;E_N(t)\right) $ is an oriented basis of $\R^N$ for every $t\in (-r,r)$, with $E_i(0)=E_i$. \\
We fix the following notation, that will be used  a lot in the paper,
$$
 Q_r:=(-r,r)\times B_{\R^{N-1}}(0,r) ,
$$
where $B_{\R^{N-1}}(0,r)$ denotes the ball in $\R^{N-1}$ with radius $r$ centered at the origin.
 Provided $r>0$  small, the map $F_{y_0}: Q_r\to \O$, given by 
$$
 (t,z)\mapsto  F_{y_0}(t,z):= \gamma(t)+\sum_{i=2}^N z_i E_i(t),
$$
is smooth and parameterizes a neighborhood of $y_0=F_{y_0}(0,0)$. We consider $\rho_\G:\G\to \R$ the distance function to the curve given by 
$$
\rho_\G(y)=\min_{\ov y\in \R^N}|y-\ov y|.
$$
In the above coordinates, we have 
\begin{equation}\label{eq:rho_Gamm-is-mod-z}
\rho_\G\left(F_{y_0}(x)\right)=|z| \qquad\textrm{ for every $x=(t,z)\in Q_r.$} 
\end{equation}
Clearly, for every $t\in (-r,r)$ and $i=2,\dots N$, there  are real numbers $\k_i(t)$ and $\tau^j_i(t)$  such that
\begin{equation}\label{DerivE_i}
E_i^\prime(t)= \kappa_i(t) \gamma^\prime(t)+\sum_{ {j=2} }^N\tau_i^j(t) E_j(t).
\end{equation}
The quantity  $\kappa_i(t)$ is the curvature in the $E_i(t)$-direction while $\tau_i^j(t)$ is the torsion from the osculating plane spanned by $\{\g'(t); E_j(t)\}$ in the direction  $E_i $.  We note that provided $r>0$ small, $\k_i$ and $\t_i^j$ are smooth functions on $(-r,r)$. Moreover, it is easy to see that 
\be\label{eq:tau-antisymm}
\t^j_i(t)=-\t^i_j(t) \qquad\textrm{ for $i,j=2,\dots, N$.   } 
\ee
The curvature vector is $\k:\G\to \R^N$  is defined as  $\k(\g(t)):=\sum_{i=2}^N \k_i(t) E_i(t)$ and its norm is given by  $|\k\g(t)|:=\sqrt{\sum_{i=2}^N\k^2_i(t)}$.
Next, we derive the expansion of the metric induced by the parameterization $F_{y_0}$ defined above.
For $x=(t,z) \in Q_r$, we define 
$$
g_{11}(x)=   {\de_t F_{y_0}}(x) \cdot   {\de_t F_{y_0}} (x) , \qquad g_{1i}(x)=   {\de_t F_{y_0}} (x)\cdot   {\de_{z_i} F_{y_0}}(x) ,\qquad  g_{ij}(x)=   {\de_{z_j} F_{y_0}}(x) \cdot   {\de_{z_i} F_{y_0}}(x).
$$
We have the following result.
\begin{lemma}\label{MaMetric}
There exits $r>0$, only depending on $\G$ and $N$, such that  for ever $x=(t,z)\in Q_r$ 
\begin{equation}\label{Vert}
\begin{cases}
\displaystyle g_{11}(x)=1+2\sum_{i=2}^Nz_i \kappa_i(0)+2t\sum_{i=2}^Nz_i \kappa^\prime_i(0)+\sum_{ij=2}^N z_i z_j \kappa_i(0) \kappa_j(0)+\sum_{ij=2}^N z_i z_j  \b_{ij}(0)+O\left(|x|^3\right)\\
\displaystyle g_{1i}(x)=\sum_{j=2}^N z_j \tau^i_j(0)+t\sum_{j=2}^N z_j \left(\tau^i_j\right)^\prime(0)+O\left(|x|^3\right)\\
\displaystyle g_{ij}(x)=\d_{ij},
\end{cases}
\end{equation}
where  $\b_{ij}(t):=\sum_{l=2}^N \tau_i^l(t) \tau_j^l(t).$
\end{lemma}
\begin{proof}
To alleviate the notations, we will write $F=F_{y_0}$.
We have 
\begin{equation}\label{Derivatives}
 {\de_t F}(x) =\gamma^\prime(t)+\sum_{j=2}^N z_j E_j^\prime(t)
\qquad \textrm{and}\qquad
 {\de_{z_i} F}(x) = E_i(t).
\end{equation}
Therefore
\begin{equation}\label{gij}
g_{ij}(x)=E_i(t)\cdot E_j(t)=\d_{ij}.
\end{equation}
By \eqref{DerivE_i} and \eqref{Derivatives}, we have
\begin{equation}\label{g_i1}
g_{1i}(x) = \sum_{l=2}^N z_l E_l^\prime(t)\cdot E_i(t)  =\sum_{j=2}^N z_j \tau_j^i(t)
\end{equation}
and
\begin{equation}\label{g_11}
g_{11}(x)= {\de_t F}(x) \cdot {\de_t F}(x) =1+2\sum_{i=2}^N z_i \kappa_i(t)+\sum_{ij=2}^N  z_i z_j \kappa_i(t)\kappa_j(t)
+\sum_{ij=2}^N z_i z_j \left(\sum_{l=2}^N \tau_i^l(t) \tau_j^l(t)\right).
\end{equation}
By   Taylor expansions, we get
$$
\kappa_i(t)=\kappa_i(0)+t\kappa_i^\prime(0)+O\left(t^2\right) 
\qquad \textrm{and}\qquad
\tau_i^k(t)=\tau_i^k(0)+t \left(\tau_i^k\right)^\prime(0)+O\left(t^2\right).
$$
Using these identities in      \eqref{g_11} and \eqref{g_i1},   we get \eqref{Vert}, thanks to   \eqref{gij}. This ends the proof.
\end{proof}
As a consequence we have the following result.
\begin{lemma}\label{MaMetricMetric}
There exists $r>0$ only depending on $\G$ and $N$, such that for every $x\in Q_r$, we have 
\begin{equation}\label{DeterminantMetric}
\sqrt{|g|}(x)=1+\sum_{i=2}^N z_i \kappa_i(0)+t\sum_{i=2}^N z_i \kappa_i^\prime(0)+\frac{1}{2}\sum_{ij=2}^N z_i z_j \kappa_i(0) \kappa_j(0)+O\left(|x|^3\right),
\end{equation}
where  $|g|$ stands for the determinant of $g$.
Moreover  $g^{-1}(x)$, the matrix  inverse   of $g(x)$,   has components given by
\begin{equation}\label{InverseMetric}
\begin{cases}
\displaystyle g^{11}(x)=1-2\sum_{i=2}^N z_i \kappa_i(0)-2t\sum_{i=2}^N z_i \kappa_i^\prime(0)+3\sum_{ij=2}^N z_i z_j \kappa_i(0) \kappa_j(0)+O\left(|x|^3\right)\\
\displaystyle g^{i1}(x)=-\sum_{j=2}^N z_j \tau^i_j(0)-t\sum_{j=2}^N z_j \left(\tau^i_j\right)^\prime(0)+2\sum_{j=2}^N z_l z_j \kappa_l(0) \tau^i_j(0)+O\left(|x|^3\right)\\
\displaystyle g^{ij}(x)=\d_{ij}+\sum_{lm=2}^N z_l z_m \tau^j_l(0) \tau^i_m(0)+O\left(|x|^3\right).
\end{cases}
\end{equation}
 
\end{lemma}
\begin{proof}
We write
$$
g(x)=id+H(x),
$$
where $id$ denotes the  identity matrix on $\R^N$  and $H $ is a symmetric matrix with components $H_{\a\b}$, for $\a,\b=1,\dots,N$, given  by
\be \label{eq:g-eq-id-H}
\begin{cases}
\displaystyle H_{11}(x)=2\sum_{i=2}^N z_i \kappa_i(0)+2t\sum_{i=2}^N z_i \kappa^\prime_i(0) +\sum_{ij=2}^N z_i z_j \kappa_i(0) \kappa_j(0)+\sum_{ij=2}^N z_i z_j \b_{ij}(0)+O\left(|x|^3\right)\\
\displaystyle H_{1i}(x)= \sum_{j=2}^N z_i \tau^i_j(0)+O\left(|x|^2\right)\\
H_{ij}(x)=0.
\end{cases}
\ee
We recall that as $|H| \to0$,
\begin{equation}\label{Exp}
\sqrt{|g|}=\sqrt{\det\left(I+H\right)}=1+\frac{\tr H}{2}+\frac{\left(\tr H\right)^2}{4}-\frac{\tr (H^2)}{4}+O\left(|H|^3\right).
\end{equation}
Now  by \eqref{eq:g-eq-id-H},  as $|x|\to 0$, we have
\be \label{eq:trH-ov-2}
\frac{\tr H}{2}=\sum_{i=2}^Nz_i \kappa_i(0)+t\sum_{i=2}^Nz_i \kappa^\prime_i(0)+\frac{1}{2}\sum_{ij=2}^N z_i z_j \kappa_i(0) \kappa_j(0)+\frac{1}{2}\sum_{ij=2}^N z_i z_j  \b_{ij}(0)+O\left(|x|^3\right),
\ee
so that  
\begin{equation}\label{l2}
\frac{\left(\tr H\right)^2}{4}=\sum_{ij=2}^N z_i z_j \kappa_i(0) \kappa_j(0)+O\left(|x|^3\right).
\end{equation}
Moreover, from \eqref{eq:g-eq-id-H}, we deduce that
$$
\tr (H^2)(x)=\sum_{\a=1}^N \left(H^2(x)\right)_{\a\a}=\sum_{\a\b=1}^N H_{\a \b}(x) H_{\b\a}(x) =\sum_{\a\b=1}^N  H^2_{\a \b}(x) =H^2_{11}(x)+2\sum_{i=2}^N H^2_{i1}(x),
$$
so that
\begin{equation}\label{l3}
-\frac{\tr (H^2)}{4}=-\sum_{ij=2}^N z_i z_j \kappa_i(0) \kappa_j(0)-\frac{1}{2} \sum_{ijl=2}^N z_i z_j \tau^l_i(0) \tau^l_j(0)+O\left(|x|^3\right).
\end{equation}
Therefore plugging the expression from \eqref{eq:trH-ov-2},    \eqref{l2} and \eqref{l3} in \eqref{Exp},   we get
$$
\sqrt{|g|}(x)=1+\sum_{i=2}^N z_i \kappa_i(0)+t\sum_{i=2}^N z_i \kappa_i^\prime(0)+\frac{1}{2}\sum_{ij=2}^N z_i z_j \kappa_i(0) \kappa_j(0)+O\left(|x|^3\right).
$$
The proof of \eqref{DeterminantMetric} is thus finished.\\

By Lemma \ref{MaMetric} we can write 
$$
g(x)=id+A(x)+B(x)+O\left(|x|^3\right),
$$
where   $A$ and $B$ are symmetric matrix with  components $(A_{\a\b})$ and  $(A_{\a\b})$, $\a,\b=1,\dots,N$, given respectively by  
\begin{equation}\label{eq:HerA}
\displaystyle A_{11}(x)=2\sum_{i=2}^N z_i \kappa_i(0),\qquad A_{i1}(x)=\sum_{j=2}^N z_j \tau^i_j(0) \qquad \textrm{and} \qquad A_{ij}(x)=0 
\end{equation}
and 
\begin{equation}\label{eq:HerB}
\begin{cases}
\displaystyle  B_{11}(x)=2t\sum_{i=2}^N z_i \kappa^\prime(0)+\sum_{i=2}^N z_i z_j\kappa_i(0)\kappa_j(0)+\sum_{ij=2}^N z_i z_j \b_{ij}(0)\\
\displaystyle B_{i1}(x)=t\sum_{j=2} z_j \left(\tau^i_j\right)^\prime(0) \qquad\textrm{and} \qquad B_{ij}(x)=0.
\end{cases}
\end{equation}
We observe that, as $|x|\to 0$, we have   
$
g^{-1}(x)=id-A(x)-B(x)+A^2(x)+O\left(|x|^3\right).
$
%
We then deduce from \eqref{eq:HerA} and \eqref{eq:HerB} that 
\begin{align*}
g^{11}(x)
&\displaystyle=1-A_{11}(x)-B_{11}(x)+A_{11}^2(x)+\sum_{i=1}^N A_{1i}^2(x)+O\left(|x|^3\right) \nonumber\\
&\displaystyle=1-2\sum_{i=2}^N z_i \kappa_i(0)-2t\sum_{i=2}^N z_i \kappa^\prime(0)+3\sum_{i=2}^N z_i z_j\kappa_i(0)\kappa_j(0)+3\sum_{ij=2}^N z_i z_j \b_{ij}(0)+O\left(|x|^3\right),
\end{align*}

\begin{align*}
g^{i1}(x)&\displaystyle=-A_{1i}(x)-B_{1i}(x)+\sum_{\a=1}^N A_{i\a} A_{1\a}+O\left(|x|^3\right)\hspace{8cm} \nonumber\\
&\displaystyle=-A_{1i}(x)-B_{1i}(x)+A_{i1}(x) A_{11}(x)+\sum_{j=2}^N A_{ij}(x) A_{1j}(x)+O\left(|x|^3\right) \nonumber\\
&\displaystyle=-\sum_{j=2}^N z_j \tau^i_j(0)-t\sum_{j=2} z_j \left(\tau^i_j\right)^\prime(0)+2\sum_{jl=2}^N z_l z_j \kappa_l(0) \tau^i_j(0)
\end{align*}

and 
\begin{align*}
g^{ij}(x)&\displaystyle=\d_{ij}-A_{ij}(x)-B_{ij}(x)+\left(A^2\right)_{ij}(x)+O\left(|x|^3\right) \hspace{8cm}  \nonumber\\
&\displaystyle=\d_{ij}-A_{ij}(x)-B_{ij}(x)+A_{1i} A_{1j}+\sum_{l=2}^N A_{il}(x) A_{jl}(x)+O\left(|x|^3\right) \nonumber\\
&\displaystyle =\d_{ij}+\sum_{{lm=2} }^N z_l z_m \tau^i_m(0) \tau^j_l(0)+O\left(|x|^3\right).
\end{align*}
This ends the proof.
\end{proof}

We recall that  the  best constant for the cylindrical Hardy-Sobolev inequality  is given by 
$$
S_{N,\s}= \min \left\{ \frac{1}{2}\int_{\R^N} |\n w|^2 dx-\frac{1}{2^*_\s}\int_{\R^N} |z|^{-\s} |w|^{2^*_\s} dx \,:\,w\in \cD^{1,2}(\R^N),\,     \right\}.
$$
Further it is   attained  by a positive function $w\in \cD^{1,2}(\R^N)$, that satisfies the Euler-Lagrange equation
\begin{equation}\label{ExpoEA}
-\D w=|z|^{-\s} w^{2^*_\s-1} \qquad \textrm{ in } \R^N,
\end{equation}
see e.g. \cite{BT}. By  \cite{FMS}, we have the last result of this section.
\begin{lemma}
For $N \geq 3$, we have
\begin{equation}\label{eq:AE}
w(x)=w(t,z)=\theta\left(|t|, |z|\right) \qquad \textrm{ for a function} \qquad \theta:\R_+ \times \R_+ \to \R_+.
\end{equation}
Moreover, there exists two constants $0<C_1<C_2$, such that
\begin{equation}\label{DecayEstimates111}
\frac{C_1}{1+|x|^{N-2}}\leq w(x) \leq \frac{C_2}{1+|x|^{N-2}} \qquad \textrm{in } \R^N.
\end{equation}
\end{lemma}
\section{Existence of minimzers for $\mu(\O, \G, h, b)$ in dimension $N\geq 4$}\label{s:Exits-N-geq4}
We consider $\O$ a bounded domain of $\R^N$, $N\geq3$ and $\G\subset \O$ be a smooth closed curve.  For $u\in H^1_0(\O)\setminus\{0\}$, we define the functional
\be\label{eq:def-J-u} 
J\left(u\right):=\displaystyle \frac{1}{2}\int_\O |\n u |^2 dy+\frac{1}{2}\int_\O h u^2dy+\frac{1}{2+\delta} \int_\O b u^{2+\d} dy-\frac{1}{2^*_\s}\int_\O \rho^{-\s}_\G|u|^{2^*_\s} dy.
\ee
We let 
$\eta \in \calC^\infty_c\left(F_{y_0}\left({Q}_{2r}\right)\right)$ be such that
$$
0\leq \eta \leq 1 \qquad \textrm{ and }\qquad \eta \equiv 1 \quad \textrm{in }  \B_r .
$$
For $\e>0$, we consider  $u_\e: \O \to  \R$ given  by
\begin{equation}\label{eq:TestFunction-w}
u_\e(y):=\e^{\frac{2-N}{2}} \eta(F^{-1}_{y_0}(y)) w \left(\e^{-1} {F^{-1}_{y_0}(y)}  \right).
\end{equation}
In particular,  for every $x=(t,z)\in \R\times \R^{N-1}$, we have 
\begin{equation}\label{eq:TestFunction-th}
u_\e\left(F_{y_0}(x)\right):=\e^{\frac{2-N}{2}}\eta\left(x\right)\th \left( {|t|}/{\e}, {|z|}/{\e} \right).
\end{equation}
It is clear that $u_\e \in H^1_0(\O).$ Then we have the following
\begin{proposition}\label{Prop1}
For all $N \geq 4$, we have
\begin{equation}\label{Expansion}
J(u_\e)=S_{N, \s}+\e^{2-\frac{\d(N-2)}{2}} b(y_0)\int_{\R^N} w^{\d+2} dx+o\left(\e^{2-\frac{\d(N-2)}{2}}\right),
\end{equation}
as $\e \to 0$.
\end{proposition}
The proof of Proposition \ref{Prop1} is divided in two parts, Lemma \ref{Lem1} and Lemma \ref{Lem2} below. For that we set
$$
J_1\left(u\right):=\displaystyle \frac{1}{2}\int_\O |\n u |^2 dx  + \frac{1}{2}\int_\O h u^2 dx-\frac{1}{2^*_\s}\int_\O \rho^{-\s}_\G|u|^{2^*_\s} dx,
$$
the following is due to the second author and Fall  \cite{Fall-Thiam}.
\begin{lemma}\label{Lem1}
We have
\begin{align}\label{eq:expans-J-u-eps}
J_1\left(u_\e\right)=S_{N,\s}+
\begin{cases}
 O(\e^2) &\qquad \textrm{for all $N \geq 5$}\\\\
 O(\e^2 |\log(\e)|) &\qquad \textrm{for all $N=4$}.
\end{cases}
\end{align} 
\end{lemma}
We finish the proof by the following
\begin{lemma}\label{Lem2}
 We have
$$
\begin{cases}
 \displaystyle\int_\O b u_\e^{2+\d}dx=\displaystyle\e^{2-\frac{\d(N-2)}{2}}b(y_0)\int_{\R^N} w^{\d+2} dx+O\left(\e^{2}\right) &\qquad \textrm{ for $N \geq 4$}\\\\
\displaystyle\int_\O b u_\e^{2+\d}dx=\displaystyle \e^{2-\frac{\d}{2}}b(y_0) \int_{Q_{r/\e}} w^{\d+2} dx+O\left(\e^{2}\right) &\qquad\textrm{ for $N=3$ and $\delta\leq 1$}\\\\
\displaystyle\int_\O b u_\e^{2+\d}dx=\displaystyle \e^{2-\frac{\d}{2}}b(y_0) \int_{\R^N} w^{\d+2} dx+O\left(\e^{1+\frac{\d}{2}}\right) &\qquad\textrm{ for $N=3$ and $\delta>1$}\\\\
\end{cases}
$$
as $\e \to 0$.
\end{lemma}
\begin{proof}
We have
$$
\int_{\O} b(x) u_\e^{2+\d} dx=\int_{F_{y_0}(Q_r)} b(x) u_\e^{2+\d} dx+\int_{F_{y_0}(Q_{2r}) \setminus F_{y_0}(Q_r)} b(x) u_\e^{2+\d} dx.
$$
Since $b$ is continuous and $r$ is small, then by the change of variable formula $y=\frac{F(x)}{\e}$, we have
\begin{align*}
\int_{\O} b(x) u_\e^{2+\d} dx&=b(y_0) \e^{2-\frac{\d(N-2)}{2}} \int_{Q_{r/\e}}  w^{2+\d} dx\\\
&+O\left(\e^{4-\frac{\d(N-2)}{2}} \int_{Q_{r/\e}} |x|^2 w^{2+\d} dx+\e^{2-\frac{\d(N-2)}{2}}\int_{Q_{2r/\e} \setminus Q_{r/\e}} w^{2+\d} dx\right)\\\
&=b(y_0) \e^{2-\frac{\d(N-2)}{2}} \int_{Q_{r/\e}}  w^{2+\d} dx\\\
&+O\left(\e^{4-\frac{\d(N-2)}{2}} \int_{Q_{r/\e}} |x|^2 w^{2+\d} dx+\e^{2-\frac{\d(N-2)}{2}}\int_{Q_{2r/\e} \setminus Q_{r/\e}} w^{2+\d} dx\right)
\end{align*}
Thanks to \eqref{DecayEstimates111}, we have
$$
\e^{4-\frac{\d(N-2)}{2}} \int_{Q_{r/\e}} |x|^2 w^{2+\d} dx+\e^{2-\frac{\d(N-2)}{2}}\int_{ Q_{2r/\e}\setminus Q_{r/\e}} w^{2+\d} dx=O(\e^2) \quad \textrm{ for all } N \geq 3
$$
and 
$$
\e^{2-\frac{\d(N-2)}{2}}\int_{Q_{2r/\e} \setminus Q_{r/\e}} w^{2+\d} dx=O(\e^2) \qquad \textrm{ for all  $N \geq 4$}.
$$
We finish by noticing that, for $N=3$, we have
$$
\int_{\R^N \setminus Q_{r/\e}} w^{2+\d} dx= O(\e^{\d-1}).
$$
This then ends the proof of the Lemma.
\end{proof}
\section{Existence of minimizer for $\mu_h(\O,\G, h, b)$ in dimension three}\label{s:3D-case}
We consider the function
$$
\cR:\R^3\setminus\{0\}\to \R,    \qquad x\mapsto  \cR(x)=\frac{1}{|x|}
$$
which  satisfies
\begin{equation}\label{eq:Green-R3-3D}
-\D \cR=0 \qquad \textrm{ in $\R^3\setminus\{0\}$. }  
\end{equation}
We denote by $G$ the solution to the equation 
\begin{equation}\label{eq:Green-3D}
\begin{cases}
-\D_x G(y, \cdot)+h G(y,\cdot)=0& \qquad \textrm{  in $\O\setminus \{y\}$. }  \\
G(y,\cdot )=0&   \qquad \textrm{  on $\de \O $, }
\end{cases}
\end{equation}
and satisfying
\be\label{eq:expand-Green-trace}
G(x,y)=  \cR(x-y)+O(1)\qquad\textrm{ for $x, y\in \O$ and $x\not= y$.}
\ee
We note that $G$ is proportional to the Green function of $-\D+h$ with zero Dirichlet data.\\
We let $\chi\in C^\infty_c(-2,2)$ with $\chi \equiv 1$ on $(-1,1)$ and $0\leq \chi<1$. For $r>0$,   we  consider the cylindrical symmetric cut-off function
\be\label{eq:def-cut-off-cylind} 
\eta_r(t,z)=\chi\left(\frac{|t|+|z|}{r} \right) \qquad \qquad\textrm{ for every  $(t,z)\in \R\times \R^2$}.
\ee
It is clear that 
$$
\eta_r\equiv 1\quad \textrm{ in $\B_r$},\qquad \eta_r\in H^1_0({Q}_{2r}),\qquad |\n \eta_r|\leq  \frac{C}{r} \quad\textrm{ in $\R^3$}.
$$
For $y_0\in \O$, we let  $r_0\in (0,1)$  such that   
\be\label{eq:def-r0} 
y_0+ Q_{2r_0}\subset\O. 
 \ee  
We define the function $M_{y_0}: Q_{2r_0}\to \R$ given by
\begin{equation}\label{C9}
M_{y_0}(x):=  G (y_0,x+y_0)-{\eta_r}(x)\frac{1}{|x|}     \qquad \textrm{ for every $x\in Q_{2r_0}$}.
\end{equation}
 It follows from  \eqref{eq:expand-Green-trace} that  $M_{y_0}\in  L^\infty(Q_{r_0})$. By \eqref{eq:Green-3D} and \eqref{eq:Green-R3-3D}, 
 $$
|-\D {M}_{y_0}(x)+h(x) {M}_{y_0}(x)|\leq \frac{C}{|x|}= C \cR(x) \qquad \textrm{ for every  $x\in Q_{r_0}$},
 $$
 whereas $\cR\in L^p( Q_{r_0})$ for every $p\in (1,3)$. Hence by
elliptic regularity theory, $M_{y_0}\in W^{2,p}(Q_{r_0/2})$ for every $p\in (1,3)$. Therefore by Morrey's embdding theorem, we deduce that 
\be \label{eq:regul-beta}
\|M_{y_0}\|_ {C^{1,\varrho}(Q_{r_0/2})}\leq C \qquad \textrm{ for every $\varrho\in (0,1)$.}
\ee
In view of \eqref{eq:expans-Green}, the mass of the operator $-\D+h$ in $\O$ at the point $y_0\in   \O$ is given by  
\be \label{eq:def-mass}
  \textbf{m}(y_0)={M}_{y_0}(0).
\ee
We recall that the positive  ground state solution $w$  
satisfies 
\begin{equation}\label{eq:w-gorund-3D}
-\D w=|z|^{-\s} w^{2^*_\s-1} \qquad \textrm{in } \R^3,
\end{equation}
where $x=(t,z)\in \R \times \R^{2}$. In addition by \eqref{DecayEstimates111}, we have 
\begin{equation}\label{eq:up-low-bound-w-3D}
\frac{C_1}{1+|x|} \leq w(x) \leq \frac{C_2}{1+|x|} \qquad \textrm{ in $  \R^3$.}
\end{equation}
The following result will be crucial in the sequel.
\begin{lemma}\label{lem:v-to-cR}
Consider the function $v_\e:  \R^3\setminus\{0\}\to \R$ given by 
$$
v_\e(x)= \e^{-1} w\left(\frac{x}{\e}\right).
$$
Then there exists a   constant $\textbf{c}>0$ and a sequence $(\e_n)_{n\in \N}$    (still denoted by $\e$)  such that 
$$
v_\e (x) \to  \frac{\textbf{c}}{|x|}\qquad \textrm{ for  all most every  $x\in \R^3 $ } 
$$
and 
\be\label{eq:nv-eps-to-nv-C1}
v_\e (x) \to  \frac{\textbf{c}}{|x|}   \qquad \textrm{ for  every   $x\in \R^3\setminus\{z=0\}$.  } 
\ee
\end{lemma}
For a proof, see for instance [\cite{Fall-Thiam}, Lemma 5.1]
%
%
%
%
%
 %
 
\qquad Next, given $y_0\in \G\subset\O\subset \R^3$, we let $r_0$ as defined in \eqref{eq:def-r0}. For    $r\in (0, r_0/2)$,  we consider $F_{y_0}: Q_r\to \O$ (see Section \ref{s:Geometric-prem}) parameterizing a neighborhood of $y_0$ in $\O$, with the property that $F_{y_0}(0)=y_0$.
For   $\e>0$,  we consider  $u_\e: \O \to  \R$ given  by
$$
u_\e(y):=\e^{-1/2} \eta_r(F^{-1}_{y_0}(y)) w \left(\frac{F^{-1}_{y_0}(y)}{\e} \right).
$$
We can now define the test function $\Psi_\e:\O\to \R$ by
\be 
\label{eq:TestFunction-Om-3D}
\Psi_\e\left(y\right)=u_\e(y)+\e^{1/2}  \textbf{c}\,  \eta_{2r}(F^{-1}_{y_0}(y) ){M}_{y_0}(F^{-1}_{y_0}(y) ).
\ee
It is plain that $\Psi_\e\in H^1_0(\O)$ and 
$$
\Psi_\e\left(F_{y_0}(x)\right)=\e^{-1/2} \eta_r(x) w \left(\frac{x}{\e} \right)+\e^{1/2}  \textbf{c} \,  \eta_{2r}(x) {M}_{y_0}(x) \qquad\textrm{ for every $x\in \R^N$.}
$$
The main result of this section is contained in the following result.
\begin{proposition}\label{Expansion-no-h}
Let    $(\e_n)_{n\in \N}$ and $\textbf{c}$ be the sequence and the number given by Lemma \ref{lem:v-to-cR}. Then there exists  $r_0,n_0>0$ such that for every $r\in (0,r_0)$ and $n\geq n_0$
 \begin{align*}
 \begin{cases}
\displaystyle  J(\Psi_\e)=  S_{3,\s}-  \e_n \pi^2 \textbf{m}(y_0)\textbf{c}^2+\frac{\e_n^{2-\frac{\d}{2}}}{2+\d} \int_{Q_{r/\e}} w^{2+\d} dx+\calO_r(\e_n)&\qquad \textrm{ for  $\d\leq 1$}\\\\
\displaystyle  J(\Psi_\e)=  S_{3,\s}-  \e_n \pi^2 \textbf{m}(y_0)\textbf{c}^2+\frac{\e_n^{2-\frac{\d}{2}}}{2+\d} \int_{\R^3} w^{2+\d} dx+\calO_r(\e_n)&\qquad \textrm{ for  $\d> 1$},
\end{cases}
\end{align*}
for some numbers   $\calO_r(\e_n)$ satisfying    
$$
\lim_{r\to 0}\lim_{n\to \infty}  \e^{-1}_n \calO_r(\e_n)=0.
$$
\end{proposition}
The proof of this proposition will be separated  into two steps given by Lemma \ref{lem:expans-num-3D} and Lemma \ref{lem:expans-denom-3D} below.
To alleviate the notations, we will write $\e$ instead of $\e_n$ and  we will remove the subscript  $y_0$, by writing $M$ and $F$ in the place of ${M}_{y_0}$ and $F_{y_0}$ respectively.
We define 
$$
\ti \eta_r(y):=\eta_r(F^{-1}(y)),\qquad V_\e(y):=v_\e(F^{-1}(y)) \qquad\textrm{ and } \qquad \ti M_{2r}(y):=\eta_{2r}(F^{-1}(y) )  M ( F^{-1}(y))  ,
$$
where $v_\e(x)=\e^{-1} w\left(\frac{x}{\e} \right).$ With these notations, \eqref{eq:TestFunction-Om-3D} becomes
\be  \label{eq:def-W-eps-3D}
\Psi_\e (y) = u_\e(y)+ \e^{\frac{1}{2}}  \textbf{c}\, \ti M_{2r}(y)= \e^{\frac{1}{2}}   V_\e(y)+    \e^{\frac{1}{2}}  \textbf{c}\, \ti M_{2r}(y) .
\ee
We first  consider  the numerator in \eqref{Expansion-no-h}.
\begin{lemma}\label{lem:expans-num-3D}
 We have
 \begin{align*}
\displaystyle J_1(\Psi_\e) 
&= S_{3,\s}-  \e \pi^2  \textbf{c}^2 \textbf{m}(y_0)   +\calO_r(\e),
\end{align*}
as $\e \to 0$.
\end{lemma}
For a proof, see for instance [\cite{Fall-Thiam}, Proposition 5.3].
The following result together with the previous lemma provides the proof of Proposition\ref{Expansion-no-h}.
\begin{lemma}\label{lem:expans-denom-3D}
We have
$$
\int_{\O} |\Psi_{\e}|^{2+\d} dy=\e^{2-\frac{\d}{2}} b(y_0)\int_{Q_{r/\e}} w^{2+\d} dx+ o\left(\e^{2-\frac{\d}{2}}\right) ,
$$
as $\e \to 0$.
\end{lemma}
\begin{proof}
Since $\delta>0$, by Taylor expansion we have
\begin{align}\label{eq:expan-L-2-star-1}
\int_{\O} |\Psi_{\e}|^{2+\d} dy&=\int_{\O} |u_\e+\e^{1/2}\ti {M}_{2r} |^{2+\d} dy\nonumber\\
&=\int_{\O} |u_\e|^{2+\d} dy+O\left(\e^{1/2} \int_{\O} |u_\e|^{1+\d}|\ti {M}_{2r} | dy+\int_{\O} |u_\e|^\d |\ti {M}_{2r} |^2 dy+\int_{\O} |\ti {M}_{2r} |^{2+\d} dy\right).
\end{align}
Using H\"{o}lder's inequality and \eqref{DeterminantMetric}, we have
\begin{align}\label{eq:expan-L-2-star-2}
\int_{F\left(\B_{ 4r}\right)} |\eta u_{\e}|^{\d} \left({\e}^{1/2}\ti M_r \right)^2 dy& \leq \e \|u_{\e}\|_{L^{2+\d}(F(\B_{4 r} ))}^{{\d}}\|\ti{M}_{2r}\|_{L^{2+\d}(F(\B_{4 r} ))}^{{2} }\nonumber\\
&=\e^{4-\frac{\d}{2}} \|w\|_{L^{2+\d}( \B_{ 4r};\sqrt{|g|} )}^{{\d}}  \|\ti{M}_{2r}\|_{L^{2+\d}(F(\B_{4 r} ))}^{{2} }\nonumber\\
&\leq \e^{4-\frac{\d}{2}}\|\ti{M}_{2r}\|_{L^{2+\d}(F(\B_{ 4r} ))}^{{2} }=o(\e),
\end{align}
Since $\d>0$, by \eqref{eq:regul-beta}, we easily get 
\begin{align} \label{eq:expan-L-2-star-2-00}
\int_{F\left(\B_{4 r}\right)}|{\e}^{1/2} \ti {M}_{2r} |^{2+\d} dy=O(\e^{1+\frac{\d}{2}})=o(\e).
\end{align}
By \eqref{eq:expan-L-2-star-1}, \eqref{eq:expan-L-2-star-2-00}, \eqref{eq:expan-L-2-star-2} and Lemma \ref{Lem2}, it results
\begin{align*}
\displaystyle\int_{\O} |\Psi_{\e}|^{2+\d} dy&=  \displaystyle \int_{F(\B_{r})} | u_{\e}|^{2+\d} dy
+O\left( {\e}^{1/2} \int_{F(\B_{r})} | u_{\e}|^{1+\d}   \ti M_{2r} dy \right)+o(\e)\\\\
&= \displaystyle \e^{2-\frac{\d}{2}}b(y_0) \int_{Q_{r/\e}} w^{\d+2} dx+ O\left({\e}^{1/2} \int_{F(\B_{r})} | u_{\e}|^{1+\d}   \ti M_{2r} dy\right) +o(\e).
\end{align*}
We define  $B_\e(x):=M(\e x)  \sqrt{|g_\e|}(x) =M(\e x)  \sqrt{|g|}(\e x)$. Then
by    the change of variable $y=\frac{F(x)}{\e}$ in the above identity and recalling \eqref{DeterminantMetric}, then by oddness, we have 
\begin{align*}
{\e}^{1/2}\int_{\O} |u_\e|^{1+\d}|\ti {M}_{2r} | dy&=O\left({\e}^{3-\d/2}\int_{\B_{r/\e}} | w |^{1+\d} dx\right)=O\left({\e}^{3-\d/2}\right).
\end{align*}
Therefore
$$
\int_{\O} |\Psi_{\e}|^{2+\d} dy=\e^{2-\frac{\d}{2}} b(y_0)\int_{Q_{r/\e}} w^{2+\d} dx+ o\left(\e^{2-\frac{\d}{2}}\right),
$$
as $\e \to 0$. This then ends the proof.
\end{proof}
%
%
%
%
%
\section{Proofs of Theorem \ref{th:main1} and Theorem \ref{th:main2}}\label{Section5}
It's well known in the literature that if
\begin{equation}\label{Compacity}
\mu_\s(\O, \G, h, b)< S_{N,\s},
\end{equation}
then $\mu_\s(\O, \G, h, b)$ is achieved by a positive function $u \in H^1_0(\O)$. For a similar result, we refer to the works of \cite{Fall-Thiam, ThiamH, Jaber4} and references therein. Therefore, the proofs of Theorem \ref{th:main1} and Theorem \ref{th:main2} are direct consequences of Proposition \ref{Prop1}, Proposition \ref{Expansion-no-h} and inequality \eqref{Compacity} above.
%
%
%
%
%
%

\end{document}